\documentclass[11pt]{article}

\usepackage{ulem}
\usepackage{framed}

\usepackage{amsthm}
\usepackage{amssymb}
\usepackage{amscd}
\usepackage{amsmath}
\usepackage[all]{xy}
\usepackage[top=30truemm,bottom=30truemm,left=25truemm,right=25truemm]{geometry}
\usepackage{comment}
\usepackage{algorithmic,algorithm}
\usepackage{here}
\usepackage{graphicx}
\usepackage{slashbox}
\usepackage{color} %
\usepackage{enumerate} %
\usepackage{listliketab}
\usepackage{multirow}
\usepackage{setspace}


\makeatletter
  
  \@addtoreset{algorithm}{section}
\makeatother
\theoremstyle{definition}

\theoremstyle{definition}

\theoremstyle{plain}

\theoremstyle{plain}

\theoremstyle{plain}

\theoremstyle{plain}

\theoremstyle{plain}
\newtheorem{thm}{Theorem}[section]
\theoremstyle{definition}

\theoremstyle{definition}

\theoremstyle{definition}

\theoremstyle{definition}
\newtheorem{defin}[thm]{Definition}
\theoremstyle{definition}

\theoremstyle{plain}
\newtheorem{prop}[thm]{Proposition}
\theoremstyle{plain}
\newtheorem{lem}[thm]{Lemma}
\theoremstyle{plain}
\newtheorem{cor}[thm]{Corollary}
\theoremstyle{definition}

\theoremstyle{definition}

\theoremstyle{definition}

\theoremstyle{definition}

\theoremstyle{definition}



\def\F{\mathbb{F}}

\def\ord{{\rm ord}}

\def\mod{{\rm mod}}

\newcommand{\bbA}{{\operatorname{\bf A}}}

\newcommand{\bbP}{{\operatorname{\bf P}}}

\newcommand{\Proj}{\operatorname{Proj}}

\newcommand{\determinant}{\operatorname{det}}

\makeatletter
\def\@seccntformat#1{\csname the#1\endcsname. }
\makeatother
\makeatletter
\renewcommand\section{\@startsection {section}{1}{\z@}%
 {-3.5ex \@plus -1ex \@minus -.2ex}%
 {2.3ex \@plus.2ex}%
 {\normalfont\large\bfseries}}
\makeatother

\begin{document}

\title%[]
{\bf The existence of supersingular curves of genus 4\\
in arbitrary characteristic}
\author
{Momonari Kudo\thanks{Kobe City College of Technology.
E-mail: \texttt{m-kudo@math.kyushu-u.ac.jp}},
\ Shushi Harashita\thanks{Graduate School of Environment and Information Sciences, Yokohama National University.
E-mail: \texttt{harasita@ynu.ac.jp}}
\ and Hayato Senda\thanks{Graduate School of Environment and Information Sciences, Yokohama National University.
E-mail: \texttt{senda-hayato-fw@ynu.jp}}}

\maketitle

%\begin{abstract}
%\end{abstract}

 \begin{abstract}
We prove that
there exists a supersingular nonsingular curve of genus $4$ 
in arbitrary characteristic $p$.
For $p>3$ we shall prove that
the desingularization of a certain fiber product over $\bbP^1$ of two supersingular elliptic curves is supersingular.
% In this paper, we explicitly determine the automorphism group of
% every superspecial nonhyperelliptic curve of genus $4$ over $\F_{11}$.
% Our algorithm determining automorphism groups works for any nonhyperelliptic curve of genus $4$ over finite fields.
% With this computation, we show the compatibility between the enumeration
% of superspecial curves of genus $4$ over $\F_{11}$
% obtained in \cite{KH17}
% and an enumeration by Galois cohomology theory.
\end{abstract}

\section{Introduction}
Let $K$ be an algebraically closed field of positive characteristic.
For a nonsingular algebraic curve $C$ over $K$
we call $C$ {\it supersingular} (resp.\ {\it superspecial}) if its Jacobian $J(C)$ is isogenous (resp.\ isomorphic) to a product of supersingular elliptic curves.
% over the algebraic closure $\overline{K}$.

As to supersingular curves, the following is a basic problem (cf.\ \cite[Question 2.2]{Pries}).
\begin{quote}
For given $g$, does there exist a supersingular curve
of genus $g$ in any characteristic $p$?
\end{quote}

For $g\le 3$, this problem was solved affirmatively.
%and $g=4$ is the first open case (cf.\ \cite[Question 3.4]{Pries}).
The case of $g=1$, i.e., elliptic curves is due to Deuring \cite{Deuring}.
As a proof for $g=2$ with $p>3$ and for $g=3$ with $p>2$, we refer to
the stronger fact that
there exists a maximal curve of genus $g$
over $\F_{p^{2e}}$ if $g=2$ and $p^{2e}\ne 4,9$
(cf.\ Serre~\cite[Th\'eor\`eme 3]{Serre1983})
and if $g=3$, $p\ge 3$ and $e$ is odd
(cf.\ Ibukiyama \cite[Theorem 1]{Ibukiyama}),
where we recall the general fact that any maximal curve over $\F_{p^2}$ is superspecial (and therefore supersingular).
Also Ibukiyama, Katsura and Oort in \cite[Proposition 3.1]{IKO}
proved the existence of superspecial curves of genus $2$ for $p>3$.
Even in characteristic $3$, there exists a supersingular curve of genus 2: for example $y^2=x^5+1$ is supersingular (but is not superspecial),
since its Cartier-Manin matrix is nilpotent,
see \eqref{Cartier-Manin matrix} and \eqref{Supersingularity} in Section \ref{SectionHoweCurves} for the Cartier-Manin matrix and a criterion for the supersingularity.
%For $g=3$, in Ibukiyama \cite[Theorem 1]{Ibukiyama}),
%  if $g=3$, $p\ge 3$ and $e$, there exists a maximal curve of genus $3$ over $\F_{p^{2}}$.
%Moreover if $g\le 3$, some theoretical approaches
%%to finding or enumerating superspecial curves
%are available,
%since any principally polarized abelian variety of dimension $g\le 3$
%is the Jacobian variety of a (possibly reducible) curve, see Oort-Ueno \cite{OU}.
%In the case of principally polarized abelian varieties,
%the number of isomorphism classes of superspecial ones is described by a class %number of
%a quaternion unitary group, see Ibukiyama-Katsura-Oort \cite[Theorem 2.10]{IKO},
%and explicit formulae of those class numbers are given by Hashimoto-Ibukiyama \%cite{HI} for $g=2$
%and by Hashimoto \cite{Hashimoto} for $g=3$.
%The enumeration of superspecial curves for $g\le 3$ is done by
%removing the contribution of reduced curves.
For the case of $p=2$, we refer to the cerebrated paper \cite{GV} by
van der Geer and van der Vlugt, where they proved that there exists a supersingular curve of an arbitrary genus in characteristic $2$.

This paper focuses on the first open case, i.e., the case of $g=4$ (cf.\ \cite[Question 3.4]{Pries}).
Let us recall some recent works, restricting ourselves to the case of $g=4$. 
In \cite{LMPT}, Li, Montovan, Pries and Tang proved
that if $p \equiv 2 \bmod{3}$ or if $p \equiv 2, 3, 4 \bmod{5}$, then there exists a supersingular curve of genus $4$, in particular for $p=3$.
%For the enumeration of superspecial curves over small finite fields, see \cite{KH16}, \cite{KH17} and \cite{KH18}.
% (cf.\ for the {\it superspecial} case, 
For odd $p \equiv 2 \bmod{3}$, in \cite{K18} the first author showed that there exists a superspecial (and thus supersingular) nonsingular curve of genus $4$.

% In the {\it superspecial} case, the (non-)existence of a superspecial curve is known for some $p$, and the isomorphism classes of such curves have been enumerated completely over finite fields for small $p$.
% Specifically, there exist superspecial nonhyperelliptic (resp.\ hyperelliptic) curves of genus $4$ for $p=5$ and $11$ (resp.\ $p = 17$ and $19$), whereas there does not exist such a curve for $p=7$.
% For $p \equiv 2 \bmod{3}$, it is shown in [??] that there exists a superspecial curve of genus $4$.
% Let us now update Ekedahl's problem for $g=4$: ``Does there exist a superspecial nonhyperelliptic (resp.\ hyperelliptic) curve for any $p > 11$ with $p \equiv 1 \bmod{3}$ (resp.\ $p > 19$)?''.
% In the {\it supersingular} case, it is deduced from [??] that there exists a supersingular curve in characteristic $p=2$.
% In [??], if $p \equiv 2 \bmod{3}$ or if $p \equiv 2, 3, 4 \bmod{5}$, then there exists a supersingular curve of genus $4$ over $\overline{\mathbb{F}_p}$, in particular for $p=3$.

This paper aims to remove any condition on $p$ for the existence of supersingular curves of genus $4$.
For this, we use curves introduced by Howe
%We now describe our main results.
in \cite{Howe}, where he studied a curve of genus $4$ defined as the desingularization of the fiber product over $\bbP^1$ of two elliptic curves.
In this paper, we call such a curve a {\it Howe curve}, see Definition \ref{DefinitionHoweCurves} for
the precise definition of Howe curves. Our main theorem is:

\begin{thm}\label{TheoremInIntroduction}
For any $p > 3$,
there exists a supersingular Howe curve of genus $4$ in characteristic $p$. 
\end{thm}

The next corollary is deduced from Theorem \ref{TheoremInIntroduction} together with the existence results above by \cite{GV} for $p=2$ and by \cite{LMPT} for $g=4$ and $p=3$.

\begin{cor}
There exists a supersingular nonsingular curve of genus $4$ in arbitrary characteristic $p>0$.
\end{cor}

As any supersingular Howe curve has $a$-number $\ge 3$ for odd $p$ (cf.\ Section \ref{SectionHoweCurves}), 
Theorem \ref{TheoremInIntroduction} is a stronger assertion than the affirmative answer for $p>3$ to the question by Pries \cite[Question 3.6]{Pries},
which predicts
that there exists a nonsingular curve of genus $4$ with $p$-rank $0$ and $a$-number at least $2$.

Let us describe an outline of the proof of Theorem \ref{TheoremInIntroduction},
with an overview of this paper.
In Section 2, we review the definition of Howe curves and their properties, 
and show that
the existence of a supersingular Howe curve of genus $4$ is equivalent to that of two supersingular elliptic curves $E_1: y^2 = f_1$ and $E_2 : y^2 = f_2$ with coprime polynomials $f_1$ and $f_2$ of degree $3$ such that the hyperelliptic curve $C : y^2 = f:=f_1 f_2$ of genus $2$ is also supersingular.
For the supersingularity of $C$, we use the fact
that any curve of genus $2$ is supersingular if and only if $M M^{\sigma} = 0$ holds for its Cartier-Manin matrix $M$, where $\sigma$ denotes the Frobenius map.
In Section 3,
for the Legendre form $y^2=g:=x(x-1)(x-t)$, we prove two key propositions on certain coefficients in $g^{(p-1)/2}$ by explicit computations.
In Section 4,
based on the two propositions, we investigate properties of entries of $M M^{\sigma}$ as polynomials, where we regard coefficients in $f_1$ and $f_2$ as variables. The properties show that we get
a desired $(f_1, f_2)$ from
a solution of a mutivariate system
obtained by removing trivial factors from $M M^{\sigma} = 0$.
Finally, we show the existense of such a solution
%we obtain the main theorem 
by proving an analogous result of the quasi-affineness of Ekedahl-Oort strata in the case of abelian varieties.

In future work, we shall enumerate the isomorphism classes of supersingular Howe curves, whereas the enumeration in the superspecial case has been done for relatively small characteristics by Senda \cite{Senda}. 
It would be meaningful to try to apply our techniques of this paper to the case of genera higher than $4$.

%Let us give an overview of this paper.
%In Section 2, we review the definition of Howe curves and some known facts shown in \cite{Howe} and discuss the supersingularity of Howe curves.
%In Section 3, the main results are stated.
%In Section 3, we prove two key propositions on elliptic curves in Legendre form.
%Based on the two propositions, Section 4 proves the main theorem.

%====================
\subsection*{Acknowledgments}
%====================
This work was supported by JSPS Grant-in-Aid for Research Activity Start-up 18H05836 and JSPS Grant-in-Aid for Scientific Research (C) 17K05196.
% \subsection*{Acknowledgments}
% The authors thank Professor Tomoyoshi Ibukiyama for his valuable comments.
% This work was supported by
% JSPS Grant-in-Aid for Scientific Research (C) 17K05196.
\section{Howe curves}\label{SectionHoweCurves}
In this section, we recall the definition of Howe curves
and properties of these curves, and study
the supersingularity of them.

\begin{defin}\label{DefinitionHoweCurves}
A {\it Howe curve} is a curve which is isomorphic to the desingularization of
the fiber product $E_1 \times_{\bbP^1} E_2$ of two double covers
$E_i\to \bbP^1$ ramified over $S_i$, where
$S_i$ consists of 4 points and $|S_1\cap S_2|=1$ holds.
%setting $S_i$ be the set of ramified points of $E_i\to \bbP^1$ for $i=1,2$.
\end{defin}

To achieve our goal, for $p>3$ we realize a Howe curve in the following way.
Let $K$ be an algebraically closed field of characteristic $p$.
% is constructed from two elliptic curves.
Let 
\begin{eqnarray}
y^2 &=& x^3 + A_1 x + B_1,\label{originalE1}\\
y^2 &=& x^3 + A_2 x + B_2\label{originalE2}
\end{eqnarray}
be two (nonsingular) elliptic curves, where $A_1,B_1,A_2,B_2\in K$.
Let $\lambda, \mu, \nu$ be elements of $K$ and set
\begin{eqnarray}
f_1(x) &=& x^3 + A_1 \mu^2 x + B_1\mu^3,\label{f1}\\
f_2(x) &=& (x-\lambda)^3 + A_2 \nu^2(x-\lambda) + B_2 \nu^3.\label{f2}
\end{eqnarray}
Consider two elliptic curves
\begin{eqnarray*}
E_1:\quad z^2y &=& y^3f_1(x/y) := x^3 + A_1 \mu^2 xy^2 + B_1\mu^3y^3,\\
E_2:\quad w^2y &=& y^3f_2(x/y) := (x-\lambda y)^3 + A_2 \nu^2(x-\lambda y)y^2 + B_2 \nu^3 y^3
\end{eqnarray*}
with the double covers 
\begin{equation*}
\pi_i : E_i \to \bbP^1=\Proj(K[x,y]).
\end{equation*}
We say that $(\lambda,\mu,\nu)$ is {\it of Howe type} if
\begin{enumerate}
\item[(i)] $\mu\ne 0$ and $\nu\ne 0$;
\item[(ii)] $f_1$ and $f_2$ are coprime.
\end{enumerate}
If $(\lambda,\mu,\nu)$ is of Howe type, then
the desingularization of the fiber product $E_1\times_{\bbP^1}E_2$
is a Howe curve, since $E_i\to\bbP^1$ is ramified over the set consisting of 4 points, say $S_i$,
and $S_1\cap S_2 = \{(1:0)\}$.

Suppose that $(\lambda,\mu,\nu)$ is of Howe type.
Put
\begin{equation}
f(x) = f_1(x)f_2(x)
\end{equation}
and consider the hyperelliptic curve $C$ of genus $2$ defined by
\begin{equation*}
C: u^2 = f(x).
\end{equation*}
%In this case, $H$ obtained in the way above are called a {\it Howe curve}.
It was proven by Howe \cite[Theorem 2.1]{Howe}
that $H$ is of genus $4$ and there exist two isogenies
\begin{eqnarray*}
\varphi: J(H) \longrightarrow E_1 \times E_2 \times J(C),\\
\psi: E_1 \times E_2 \times J(C) \longrightarrow J(H)
\end{eqnarray*}
such that $\varphi\circ\psi$ and $\psi\circ\varphi$ are the multiplication by $2$.
%$J(H)$ is isogenous to $J(E_1)\times J(E_2) \times J(C)$.
Hence $H$ is supersingular if and only if $E_1$, $E_2$ and $C$
is supersingular. Moreover, if $p$ is odd,
the $a$-number of $H$ is equal to the sum of
the $a$-numbers of $E_1$, $E_2$ and $J(C)$,
whence any supersingular Howe curve is of $a$-number $\ge 3$.

Now we recall a criterion for the supersingularity of $C$.
Let $\gamma_i$ be the $x^i$-coefficient of $f(x)^{(p-1)/2}$, i.e.,
\begin{equation}
f(x)^{(p-1)/2} = \sum_{i=0}^{3(p-1)} \gamma_i x^i.
\end{equation}
Put
\[
a=\gamma_{p-1},\quad
b=\gamma_{2p-1},\quad
c=\gamma_{p-2}\quad \text{and}\quad
d=\gamma_{2p-2}.
\]
In Section 4, we shall use the fact that $\gamma_i$ and therefore $a,b,c$ and $d$
are homogeneous when we regard them as polynomials in $\lambda, \mu$ and $\nu$.
%Let $C$ be a nonsingular curve over $k$.
Let $M$ be the Cartier-Manin matrix of $C$,
that is a matrix representing the action of Verschiebung on $H^0(C,\Omega^1)$.
It is known (cf.\ \cite[4.1]{Gonzalez} and \cite[\S 2]{Yui}) that the Cartier-Manin matrix of $C$ is given by
\begin{equation}\label{Cartier-Manin matrix}
M := \begin{pmatrix}
a & b \\
c & d
\end{pmatrix}.
\end{equation}
It follows from \cite[Step 2 of the proof of Theorem 4.8]{CO}
(also see \cite[4.0.3]{H07} for another proof)
that a curve $D$ of genus $2$ is supersingular if and only if
the Verschiebung $V$ on $H^0(D,\Omega_D^1)$ satisfies $V^2=0$.
Hence $C$ is supersingular if and only if
\begin{equation}\label{Supersingularity}
M M^\sigma = \begin{pmatrix}
a^{p+1}+bc^p &  ab^p+bd^p\\
a^pc+c^pd & b^pc+d^{p+1}
\end{pmatrix} = 0,
\end{equation}
where $M^\sigma = \begin{pmatrix}a^p & b^p \\ c^p & d^p\end{pmatrix}$.

\begin{prop}\label{EquationsOfSupersingularity}
Assume that $E_1$ and $E_2$ are supersingular.
Then $H$ is supersingular if and only if
$ad-bc=0$, $ab^{p-1}+d^p=0$ and $a^p+c^{p-1}d=0$.
\end{prop}
\begin{proof}
Since $E_1$ and $E_2$ are supersingular,
$H$ is supersingular if and only if $C$ is supersingular.
As explained above, $C$ is supersingular if and only if
\eqref{Supersingularity} holds.
%\begin{equation}\label{Supersingularity}
%M M^\sigma = \begin{pmatrix}
%a^{p+1}+bc^p &  ab^p+bd^p\\
%a^pc+c^pd & b^pc+d^{p+1}
%\end{pmatrix} = 0.
%\end{equation}
If $ad-bc=0$, then we have
\begin{eqnarray}\label{four entries}
a^{p+1}+bc^p &=& a(a^p+c^{p-1}d),\\
ab^p+bd^p &=& b(ab^{p-1}+d^p),\\
a^pc+c^pd &=& c(a^p+c^{p-1}d),\\
b^pc+d^{p+1} &=& d(ab^{p-1}+d^p).
\end{eqnarray}
Thus the ``if"-part is true.

Suppose that \eqref{Supersingularity} holds.
Then clearly we have
$\determinant(M)=ad-bc=0$, which implies
\begin{eqnarray}
a^{p-1}(ab^{p-1}+d^p) &=& b^{p-1}(a^p+c^{p-1}d),\label{BasicEqualities1}\\
c^{p-1}(ab^{p-1}+d^p) &=& d^{p-1}(a^p+c^{p-1}d).\label{BasicEqualities2}
\end{eqnarray}
Also by \eqref{four entries} we have $a^p+c^{p-1}d=0$ unless $a=c=0$ and
$ab^{p-1}+d^p=0$ unless $b=d=0$.
If $a=c=0$, by \eqref{BasicEqualities1} and \eqref{BasicEqualities2} we have $a^p+c^{p-1}d=0$
unless $b=d=0$. Similarly
If $b=d=0$, by \eqref{BasicEqualities1} and \eqref{BasicEqualities2} we have $ab^{p-1}+d^p=0$
unless $a=c=0$. Obviously $(a,b,c,d)=(0,0,0,0)$ satisfies
$ad-bc=0$, $ab^{p-1}+d^p=0$ and $a^p+c^{p-1}d=0$.
Thus the desired assertion is true.
\end{proof}

For later use, we review how the Cartier-Manin matrix is changed by a linear change of variables.
\begin{lem}\label{ChangeOfVariables}
Let $X$ be a new variable and consider substituting
$uX+v$ for $x$, where $u,v\in K$ with $u\ne 0$.
Let $\gamma'_i$ be the $X^i$-coefficient of $f(uX+v)^{(p-1)/2}$ and
set 
\[
M' = \begin{pmatrix}\gamma'_{p-1}&\gamma'_{2p-1}\\\gamma'_{p-2}&\gamma'_{2p-2}\end{pmatrix}.\]
Then we have
\[
M' = P^{-1}MP^\sigma
\]
with
\[
P = \begin{pmatrix}u & 0 \\ uv & u^2\end{pmatrix}.
\]
\end{lem}
\begin{proof}
We have
\[
\gamma'_i = \sum_{j = i}^{3(p-1)} \binom{j}{i}u^iv^{j-i}\gamma_j 
\]
as
\[
f(uX+v)^{(p-1)/2} = \sum_{j=0}^{3(p-1)} \gamma_j(uX+v)^j
%= \sum_{j=0}^{3(p-1)}\sum_{i=0}^{j} \binom{j}{i}\gamma_j u^iv^{j-i}X^i
= \sum_{i=0}^{3(p-1)}\left(\sum_{j = i}^{3(p-1)} \binom{j}{i}u^iv^{j-i}\gamma_j \right) X^i.
\]
Then we have
\begin{eqnarray*}
\gamma'_{p-1} &=& u^{p-1}(\gamma_{p-1} + v^p\gamma_{2p-1}),\\
\gamma'_{2p-1} &=& u^{2p-1}\gamma_{2p-1},\\
\gamma'_{p-2} &=& u^{p-2}(\gamma_{p-2}-v\gamma_{p-1} +v^p\gamma_{2p-2} - v^{p+1}\gamma_{2p-1}),\\
\gamma'_{2p-2} &=& u^{2p-2}(\gamma_{2p-2} - v\gamma_{2p-1})
\end{eqnarray*}
by calculating, in characteristic $p$, the binomials
\[
\binom{j}{p-1} =\begin{cases}1 & \text{if } j=p-1,2p-1,\\0 & \text{otherwise},\end{cases}
\qquad
\binom{j}{2p-1} =\begin{cases}1 & \text{if } j=2p-1,\\0 & \text{otherwise},\end{cases}\quad
\]
\[
\binom{j}{p-2} =\begin{cases}1 & \text{if } j=p-2,2p-2\\ -1 & \text{if } j=p-1,2p-1,\\ 0 & \text{otherwise},\end{cases}\qquad
\binom{j}{2p-2} =\begin{cases}1 & \text{if } j=2p-2,\\
-1 & \text{if } j=2p-1,\\
0 & \text{otherwise}\end{cases}
\]
for $j\le 3p-3$.
\end{proof}
\section{Two propositions on the Legendre form}
We show two propositions (Propositions \ref{KeyProp} and \ref{Separatedness}) on the Legendre form, which play important roles in the proof of the main theorem.

Assume $p\ge 3$.
Let $g(x)=x(x-1)(x-t)$ and $e=(p-1)/2$, where we regard $t$ as an indeterminate.
We define a polynomial $H_p (t)$ by
\begin{eqnarray}
H_p (t) &:= & \sum_{i=0}^{e} \binom{e}{i}^2 t^i. \label{Hp_t}
\end{eqnarray}
Let $\delta_{p-1}(t)$ be the $x^{p-1}$-coefficient of $g(x)^e$.
It follows from
\begin{eqnarray}
\delta_{p-1}(t) = (-1)^e  H_p(t) \label{eq:delta_p-1}
\end{eqnarray}
that
$y^2 = x(x-1)(x-t_0)$ is a supersingular elliptic curve
for $t_0 \in \overline{\F_p}$ if and only if $H_p(t_0) = 0$,
see e.g., \cite[Chap.\ V, Theorem 4.1]{Sil}.

For our purpose, we need to study
the $x^{p-2}$-coefficient $\delta_{p-2}(t)$ of $g(x)^e$.
%as well as $\delta_{p-1}(t)$.
\begin{prop}\label{KeyProp}
%Let $\delta_{p-1}(t)$ and $\delta_{p-2}(t)$ be the $x^{p-1}$-coefficient and the $x^{p-2}$-coefficient of $g(x)^e$ respectively.
%Then 
We have the following:
\begin{enumerate}
\item[$(1)$] The polynomial $\delta_{p-2}(t)$ is given explicitly by
%\begin{eqnarray}
%\delta_{p-1}(t) = (-1)^e  H_p (t), \label{eq:delta_p-1}
%\end{eqnarray}
%and
\begin{eqnarray}
\delta_{p-2}(t) = (-1)^{e-1} \sum_{i=1}^{e} \binom{e}{i-1} \binom{e}{i} t^i . \label{eq:delta_p-2}
\end{eqnarray}
%\item[$(1)$] $2 \delta_{p-2}^{\prime}(t) =  2 t \delta_{p-1}^{\prime}(t) + \delta_{p-1}(t)$.
\item[$(2)$] $\displaystyle 2 \frac{d}{dt} \delta_{p-2}(t) =  2 t \frac{d}{dt}\delta_{p-1}(t) + \delta_{p-1}(t)$.
\item[$(3)$] $\displaystyle (e + 1) \delta_{p-2}(t) = (e + 1) t \delta_{p-1}(t) + t (t - 1) \frac{d}{dt} \delta_{p-1}(t) $.
\item[$(4)$] The polynomials $\delta_{p-1}(t)$ and $\delta_{p-2}(t)$ have no common root.
\end{enumerate}
\end{prop}

\begin{proof}
\begin{enumerate}
\item It follows from the binomial theorem that
\begin{eqnarray}
g(x)^e &=& x^e(x-1)^e(x-t)^e \nonumber \\
&=& x^e \sum_{i=0}^e \binom{e}{i}x^i(-1)^{e-i}\sum_{j=0}^e \binom{e}{j}x^j(-t)^{e-j} \nonumber \\
&=& \sum_{i,j} (-1)^{i+j}\binom{e}{i}\binom{e}{j}x^{e+i+j}t^{e-j}. \label{eq:g_e}
\end{eqnarray}
% Since $i+j=e$ and $(-1)^e = 1$ (resp.\ $i+j=e-1$ and $(-1)^{e-1}=-1$) if $e+i+j=p-1$ (resp.\ $e+i+j=p-2$), we have
%Since $i+j=e$ (resp.\ $i+j=e-1$) if $e+i+j=p-1$ (resp.\ $e+i+j=p-2$), we have
Since $i+j=e-1$ if $e+i+j=p-2$, we have
%\[
%\delta_{p-1}(t) = (-1)^e \sum_{i+j=e} \binom{e}{i} \binom{e}{j} t^{e-j} = (-1)^e \sum_{i=0}^e \binom{e}{i} \binom{e}{e-i} t^{i} = (-1)^e \sum_{i=0}^{e} \binom{e}{i}^2 t^i
%\]
%and 
\begin{eqnarray}
\delta_{p-2}(t) &=& (-1)^{e-1} \sum_{i+j=e-1} \binom{e}{i} \binom{e}{j} t^{e-j} = (-1)^{e-1} \sum_{i=0}^{e-1} \binom{e}{i} \binom{e}{e-1-i} t^{i+1} \nonumber \\
& = & (-1)^{e-1} \sum_{i=0}^{e-1} \binom{e}{i} \binom{e}{i+1} t^{i+1}, \nonumber
\end{eqnarray}
which is equal to \eqref{eq:delta_p-2} by replacing the index $i$ by $i+1$.
\item The polynomial $H_p (t)$ is rearranged as follows:
\begin{eqnarray}
H_p (t) = \sum_{i=0}^e \binom{e}{i} \binom{e}{i} t^{i} & = & \sum_{i=0}^e \left( \binom{e+1}{i+1} -  \binom{e}{i+1} \right) \binom{e}{i} t^{i} \nonumber \\
& = & \sum_{i=0}^e \binom{e+1}{i+1} \binom{e}{i} t^i - \sum_{i=0}^{e} \binom{e}{i} \binom{e}{i+1} t^{i} \nonumber \\
% & = & \sum_{i=0}^e \binom{e+1}{i+1} \binom{e}{i} t^i - \sum_{i=1}^e \binom{e}{i-1} \binom{e}{i} t^{i-1} \nonumber \\
& = & \sum_{i=0}^e \binom{e+1}{i+1} \binom{e}{i} t^i - \frac{1}{t} \sum_{i=1}^e \binom{e}{i-1} \binom{e}{i} t^i . \nonumber
%  \\& = & \sum_{i=0}^e \binom{e+1}{i+1} \binom{e}{i} t^i + \frac{1}{t} \delta_{p-2}(t) \nonumber
\end{eqnarray}
% and thus
Multiplying by $(-1)^e t$ the both sides, we have
\begin{eqnarray}
t \delta_{p-1}(t) = (-1)^e \sum_{i=0}^e \binom{e+1}{i+1} \binom{e}{i} t^{i+1} + \delta_{p-2}(t), \label{eq:tc_p-1}
\end{eqnarray}
and differentiating the both sides of \eqref{eq:tc_p-1} yields the following equality
%\[
%\delta_{p-1}(t) + t \delta_{p-1}^{\prime}(t) = \sum_{i=0}^e \binom{e+1}{i+1} \binom{e}{i} (i+1) t^i + \delta_{p-2}^{\prime}(t),
%\]
\[
\delta_{p-1}(t) + t \frac{d}{dt} \delta_{p-1}(t) = (-1)^e \sum_{i=0}^e \binom{e+1}{i+1} \binom{e}{i} (i+1) t^i + \frac{d}{dt} \delta_{p-2}(t) ,
\]
the right hand side of which is equal to $ \displaystyle (e+1)\delta_{p-1}(t) + \frac{d}{dt} \delta_{p-2}(t)$
%\[
% (e+1) \sum_{i=0}^e \binom{e}{i} \binom{e}{i} t^i + c_{p-2}^{\prime} = (e+1)c_{p-1} + c_{p-2}^{\prime}.
%\]
% \[
%  (e+1) \sum_{i=0}^e \binom{e}{i} \binom{e}{i} t^i + \frac{d}{dt} \delta_{p-2}(t) = (e+1)\delta_{p-1}(t) + \frac{d}{dt} \delta_{p-2}(t).
% \]
by \eqref{eq:delta_p-1}.
Here we have
%\[
%c_{p-2}^{\prime} = - e c_{p-1} + t c_{p-1}^{\prime} = t c_{p-1}^{\prime} - e c_{p-1},
%\]
\[
\frac{d}{dt} \delta_{p-2}(t) = 
% - e \delta_{p-1}(t) + t  \frac{d}{dt} \delta_{p-1}(t) =
 t  \frac{d}{dt} \delta_{p-1}(t) - e \delta_{p-1}(t),
\]
and by multiplying by $2$, we get
%\[
%2 c_{p-2}^{\prime} = 2 t c_{p-1}^{\prime} - (p-1) c_{p-1},
%\]
\[
2 \frac{d}{dt} \delta_{p-2}(t) = 2 t \frac{d}{dt} \delta_{p-1}(t) - (p-1) \delta_{p-1}(t),
\]
whose right hand side is equal to $\displaystyle 2 t \frac{d}{dt} \delta_{p-1}(t) + \delta_{p-1}(t)$ in characteristic $p$.

\item Since $\binom{e}{i} = - \binom{e}{i-1}+  \binom{e+1}{i}$ for $1 \leq i \leq e$, it follows from \eqref{eq:delta_p-2} that
\begin{eqnarray}
(-1)^e \delta_{p-2} (t)
%  & = &  - \left( \sum_{i=1}^{e} \binom{e}{i-1} \binom{e+1}{i} t^i - \sum_{i=1}^{e} \binom{e}{i-1} \binom{e}{i-1} t^i \right) \nonumber \\
& = & \sum_{i=1}^{e} \binom{e}{i-1} \binom{e}{i-1} t^i - \sum_{i=1}^{e} \binom{e}{i-1} \binom{e+1}{i} t^i , \label{eq:delta_p-2-2} 
\end{eqnarray}
where we factor out $t$ from the first summation of \eqref{eq:delta_p-2-2}, namely
\[
\sum_{i=1}^{e} \binom{e}{i-1} \binom{e}{i-1} t^i = t \sum_{i=0}^{e-1} \binom{e}{i} \binom{e}{i} t^i = t H_p (t) - t^{e+1}.
\]
We have
\[
(-1)^e \delta_{p-2} (t) =  t H_p (t) - t^{e+1} - \sum_{i=1}^{e} \binom{e}{i-1} \binom{e+1}{i} t^i.
\]
Multiplying by $e+1$ the both sides gives
\begin{equation}\label{eq:c_p-2}
(e+1) (-1)^e \delta_{p-2} (t) = (e + 1) \left( t H_p (t) - t^{e+1} \right) - (e + 1)\sum_{i=1}^{e} \binom{e}{i-1} \binom{e+1}{i} t^i.
\end{equation}
%\begin{eqnarray}
%(e+1) (-1)^e \delta_{p-2} (t) &=& (e + 1) \left( t H_p (t) - t^{e+1} \right) - (e + 1)\sum_{i=1}^{e} \binom{e}{i-1} \binom{e+1}{i} t^i \nonumber \\
%& = & (e+1) t H_{p} (t) - \left( (e+1) t^{e+1} + (e + 1)\sum_{i=1}^{e} \binom{e}{i-1} \binom{e+1}{i} t^i \right). \label{eq:c_p-2}
%\end{eqnarray}
Changing the start of the summation in \eqref{eq:c_p-2} from $i=1$ to $i=0$, we have
\[
(e + 1)\sum_{i=1}^{e} \binom{e}{i-1} \binom{e+1}{i} t^i = (e + 1)\sum_{i=0}^{e-1} \binom{e}{i} \binom{e+1}{i+1} t^{i+1} = \sum_{i=0}^{e-1} \binom{e+1}{i+1} \binom{e+1}{i+1} (i+1)t^{i+1}
\]
and hence the right hand side of \eqref{eq:c_p-2} is
\[
(e+1) t H_p (t) - t \sum_{i=0}^{e} \binom{e+1}{i+1}^2 (i+1)t^{i}.
\]
Thus we have
\begin{eqnarray}
(e+1) (-1)^e \delta_{p-2} (t) & = & (e+1) t H_p (t) - t \sum_{i=0}^{e} \binom{e+1}{i+1}^2 (i+1)t^{i} , \label{eq:c_p-2_2}
\end{eqnarray}
and thus it suffices to show $- \sum_{i=0}^{e} \binom{e+1}{i+1}^2 (i+1)t^{i} = (t - 1) \frac{d}{dt} H_p(t)$, where
\begin{eqnarray}
& & \sum_{i=0}^{e} \binom{e+1}{i+1}^2 (i+1) t^i = \sum_{i=0}^{e} \left( \binom{e}{i} + \binom{e}{i+1} \right)^2 (i+1) t^i \nonumber \\
& = & \sum_{i=0}^{e} \binom{e}{i}^2 (i+1) t^i + \sum_{i=0}^{e} 2 \binom{e}{i} \binom{e}{i+1} (i+1) t^i + \sum_{i=0}^{e} \binom{e}{i+1}^2 (i+1) t^i . \label{eq:total}
\end{eqnarray}
In the following, we calculate each of the three summations in \eqref{eq:total}.
By \eqref{eq:delta_p-1}, the first term of \eqref{eq:total} is
\begin{eqnarray}
\sum_{i=0}^{e} \binom{e}{i}^2 (i+1) t^i & = &  \sum_{i=0}^{e} \binom{e}{i}^2 i t^i + \sum_{i=0}^{e} \binom{e}{i}^2 t^i \nonumber \\
%& = & t \sum_{i=1}^{e} \binom{e}{i}^2 i t^{i-1} + \delta_{p-1}(t) = t \delta_{p-1}^{\prime}(t) + \delta_{p-1}(t) \label{eq:1st}
& = & t \sum_{i=1}^{e} \binom{e}{i}^2 i t^{i-1} + H_p(t) = t \frac{d}{dt} H_p (t) + H_p (t) \label{eq:1st}
\end{eqnarray}
and the third one is
\begin{eqnarray}
%\sum_{i=0}^{e} \binom{e}{i+1}^2 (i+1) t^i & = & \sum_{i=0}^{e-1} \binom{e}{i+1}^2 (i+1) t^i  = c^{\prime}_{p-1} (t). \label{eq:3rd}
\sum_{i=0}^{e} \binom{e}{i+1}^2 (i+1) t^i & = & \sum_{i=0}^{e-1} \binom{e}{i+1}^2 (i+1) t^i  = \frac{d}{dt} H_p (t). \label{eq:3rd}
\end{eqnarray}
Using the equality in the statement (2), we have that the second term of \eqref{eq:total} is
\begin{eqnarray}
\sum_{i=0}^{e} 2 \binom{e}{i} \binom{e}{i+1} (i+1) t^i & = & \sum_{i=1}^{e} 2 \binom{e}{i-1} \binom{e}{i} i t^{i-1} = - 2 \frac{d}{dt} \left( (-1)^e \delta_{p-2}(t) \right) \nonumber \\
& = & - 2 t \frac{d}{dt}\left( (-1)^e \delta_{p-1}(t) \right) - (-1)^e \delta_{p-1}(t) \nonumber \\
& = & - 2 t \frac{d}{dt}H_p (t) - H_p (t).  \label{eq:2nd}
\end{eqnarray}
It follows from \eqref{eq:1st}--\eqref{eq:2nd} that \eqref{eq:total} is equal to
% we have
\begin{eqnarray}
%\sum_{i=0}^{e} \binom{e+1}{i+1}^2 (i+1) t^i & = & \left( t c_{p-1}^{\prime} + c_{p-1} \right) 
%+ \left( - 2 t c_{p-1}^{\prime} - c_{p-1} \right) + c^{\prime}_{p-1} = (1 - t) c^{\prime}_{p-1} \nonumber
% \sum_{i=0}^{e} \binom{e+1}{i+1}^2 (i+1) t^i & = & 
\left( t \frac{d}{dt} H_p (t) + H_p (t) \right) 
+ \left( - 2 t \frac{d}{dt} H_p (t) - H_p (t) \right) + \frac{d}{dt} H_p (t) & = &(1 - t) \frac{d}{dt} H_p (t), \nonumber
\end{eqnarray}
and thus $- \sum_{i=0}^{e} \binom{e+1}{i+1}^2 (i+1)t^{i} = (t - 1) \frac{d}{dt} H_p(t)$, as desired.
\if 0
From \eqref{eq:c_p-2_2}, we finally obtain
%\[
%(e + 1) c_{p-2} = (e + 1) t c_{p-1} - t \sum_{i=0}^e \binom{e+1}{i+1}^2 (i+1) t^i = (e + 1) t c_{p-1} + t (t - 1) c_{p-1}^{\prime} .
%\]
\begin{eqnarray*}
(e + 1) \delta_{p-2}(t) &=& (e + 1) t \delta_{p-1}(t) - t \sum_{i=0}^e \binom{e+1}{i+1}^2 (i+1) t^i \\
&=& (e + 1) t \delta_{p-1}(t) + t (t - 1) \frac{d}{dt} \delta_{p-1}(t) .
\end{eqnarray*}
\fi
%Hence if $\delta_{p-1}(\alpha) =\delta_{p-2}(\alpha) = 0$ for some $\alpha \in \overline{K} \smallsetminus \{ 0, 1 \}$, then one has $\delta_{p-1}^{\prime}(\alpha) = 0$, and thus $\alpha$ is a double root of $H_p (t)= \delta_{p-1}(t)$.
\item By the equality in the third statement, if $\delta_{p-1}(\alpha) =\delta_{p-2}(\alpha) = 0$ for some $\alpha \in \overline{\F_p} \smallsetminus \{ 0, 1 \}$, then one has $\frac{d}{dt} \delta_{p-1}(\alpha) = 0$.
Since $\delta_{p-1}(t) = (-1)^e H_p (t)$ by \eqref{eq:delta_p-1}, we have $H_{p}(\alpha)=\frac{d}{dt} H_{p}(\alpha) = 0$, which means that $\alpha$ is a double root of $H_p (t)$.
This contradicts the fact shown by Igusa \cite{Igusa}
that all roots of $H_p(t)$ are simple (cf.\ \cite[Chap.\ V, Theorem 4.1 (c)]{Sil}).
\end{enumerate}
\end{proof}

Before we state the second proposition (Proposition \ref{Separatedness}),
we review, for the reader's convenience, elementary congruence relations,
which is used in the proof of the second proposition.

\begin{lem}\label{lem:factorial}
We have 
\[
(p-i)! \equiv \frac{(-1)^i}{(i-1)!} \pmod{p}
\]
for $i=1,\ldots,p$.
%the following:
%\begin{enumerate}
%\item[\rm (1)] $\displaystyle \left( e! \right)^2 \equiv (-1)^{\frac{p+1}{2}} \pmod{p}$.
%\item[\rm (1)] $\displaystyle (p-i)! \equiv \frac{(-1)^i}{(i-1)!} \pmod{p}$ for $i=1,\ldots,p$.
%\item[\rm (2)] $\displaystyle (e - i)! \equiv \frac{(-1)^{-e - i -1}}{(e + i)!} \pmod{p}$ for $i=0,\ldots,e$.
%\end{enumerate}
\end{lem}

\begin{proof}
%\begin{enumerate}
%\item[\rm (1)] Since $p - i \equiv -i \pmod{p}$ for all $i$, we have
%\begin{eqnarray}
%%( p - 1 ) ! & = & (p - 1) ( p - 2) \cdots \left( p - \frac{p-1}{2} \right) \left( p - \frac{p+1}{2} \right) ! \\ \nonumber
%%& \equiv & ( - 1) ( -2) \cdots \left( -\frac{p-1}{2} \right) \left( \frac{p-1}{2} \right) ! \pmod{p} \nonumber
%( p - 1 ) ! & = & (p - 1) ( p - 2) \cdots \left( p - e \right) \cdot \left( p - e - 1 \right) !  \equiv (- 1)^{e} e! e! \pmod{p} . \label{eq:1} 
%% & \equiv & ( - 1) ( -2) \cdots \left( -e \right) e ! \pmod{p} \nonumber \\
%% & \equiv & (- 1)^{e} e! e! \pmod{p} . \label{eq:1}
%\end{eqnarray}
%\if 0
%and thus
%\begin{eqnarray}
%%(p - 1) ! \equiv (- 1)^{\frac{p-1}{2}} \left( \frac{p-1}{2} \right) ! \left( \frac{p-1}{2} \right) ! \pmod{p} . \label{eq:1}
%(p - 1) ! \equiv (- 1)^{e} e! e! \pmod{p} . \label{eq:1}
%\end{eqnarray}
%\fi
%It follows from Wilson's theorem that the left hand side of \eqref{eq:1} is equal to $-1$ modulo $p$, and hence the required claim holds.\
%\item[\rm (1)] 
We have
\[
(p - i)! = \frac{(p-1)!}{(p-1)(p-2) \cdots \left( p -i +1\right)} \equiv \frac{-1}{(-1)(-2) \cdots (-i+1 )} = \frac{(-1)^{i}}{(i-1)!} \pmod{p},
\]
where we used Wilson's theorem $(p-1)!\equiv -1 \pmod{p}$.
% where $\frac{-1}{(-1)(-2) \cdots (-i+1 )} =  \frac{(-1)^{i}}{(i-1)!} $.
% and
% \begin{eqnarray*}
% \frac{-1}{(-1)(-2) \cdots (-i+1 )}
% &=& \frac{-1}{(-1)^{i-1} \cdot 1 \cdot 2 \cdots (i-1)} = \frac{(-1)^{i}}{(i-1)!} .
% \end{eqnarray*}
%\item[\rm (2)] The right hand side of
%\[
%(e - i)! = \left(\frac{p-1}{2} - i \right)! = \frac{(p-1)!}{(p-1)(p-2) \cdots \left(p- \frac{p+1}{2} \right) \cdots \left(p - \frac{p+1}{2} -i +1\right)}
%\]
%is equal to $\frac{-1}{(-1)(-2) \cdots (- \frac{p+1}{2}) \cdots \left(- \frac{p+1}{2} -i+1 \right)}$ modulo $p$ by Wilson's theorem, and moreover we have
%\begin{eqnarray*}
%\frac{-1}{(-1)(-2) \cdots (- \frac{p+1}{2}) \cdots \left(- \frac{p+1}{2} -i+1 \right)}
%&=& \frac{(-1)^{-\frac{p+1}{2} - i }}{ \left(\frac{p+1}{2} + i-1\right)!} = \frac{(-1)^{-e - i - 1}}{(e + i)!} .
%\end{eqnarray*}
%% \begin{eqnarray*}
%% \frac{-1}{(-1)(-2) \cdots (- \frac{p+1}{2}) \cdots \left(- \frac{p+1}{2} -i+1 \right)}
%% &=& \frac{-1}{(-1)^{\frac{p+1}{2}} \cdot 1 \cdot 2 \cdots \left(\frac{p+1}{2} \right) \cdots \left(\frac{p+1}{2} + i-1 \right)} \\
%% &=& \frac{(-1)^{-\frac{p+1}{2} - i }}{ \left(\frac{p+1}{2} + i-1\right)!} = \frac{(-1)^{-e - i - 1}}{(e + i)!} .
%% \end{eqnarray*}
%\end{enumerate}
\end{proof}

\begin{prop}\label{Separatedness}
% Let $g(x)=x(x-1)(x-t)$ and $e=(p-1)/2$, where we regard $t$ as an indeterminate.
Let $H_p(t)$ be the polynomial defined by \eqref{Hp_t}.
Let $a_i \in \overline{\F_p}$ $(i=1,\ldots,e)$ be the roots of $H_p(t)=0$, i.e.,
$H_p(t) = \prod_{i=1}^{e}(t-a_i)$.
Then we have
\[
\frac{(-1)^e e!}{(p-1)!}(g(x)^e)^{(e)} = \prod_{i=1}^e\{(t-a_i)x-(1-a_i)t\},
\]
where $(g(x)^e)^{(e)}$ denotes the $e$-th derivative of $g(x)^{e}$ with respect to $x$.
\end{prop}
\begin{proof}
% \begin{eqnarray*}
% g(x)^e &=& x^e(x-1)^e(x-t)^e\\
% &=& x^e \sum_{i=0}^e \binom{e}{i}x^i(-1)^{e-i}\sum_{j=0}^e \binom{e}{j}x^j(-t)^{e-j}\\
% &=& \sum_{i,j} (-1)^{i+j}\binom{e}{i}\binom{e}{j}x^{e+i+j}t^{e-j}.
% \end{eqnarray*}
It follows from \eqref{eq:g_e} that
\[
(g(x)^e)^{(e)}=\sum_{i,j} (-1)^{i+j}\frac{(e+i+j)!}{(i+j)!}\binom{e}{i}\binom{e}{j}x^{i+j}t^{e-j},
\]
whose coefficient of $x^m t^n$ with $0\le m,n\le e$ is
\begin{eqnarray}
(-1)^m\frac{(e+m)!}{m!}\binom{e}{m+n-e}\binom{e}{e-n}. \label{eq:coef_of_left}
\end{eqnarray}
Putting $P:=\prod_{i=1}^e\{(t-a_i)x-(1-a_i)t\}$, we have
\[
P = (x-t)^e\prod_{i=1}^e\left(\frac{t(x-1)}{x-t}-a_i\right) = (x-t)^e H_p\left(\frac{t(x-1)}{x-t}\right),
\]
whose right hand side is equal to $(x-t)^e\sum_{i=0}^e \binom{e}{i}^2 \left(\frac{t(x-1)}{x-t}\right)^i$ by \eqref{eq:delta_p-1}.
Using the binomial theorem, one obtains
% \begin{eqnarray*}
% P:=\prod_{i=1}^e\{(t-a_i)x-(1-a_i)t\} &=& (x-t)^e\prod_{i=1}^e\left(\frac{t(x-1)}{x-t}-a_i\right)\\ &=& (x-t)^e H_p\left(\frac{t(x-1)}{x-t}\right)\\
% &=& (x-t)^e\sum_{i=0}^e \binom{e}{i}^2 \left(\frac{t(x-1)}{x-t}\right)^i\\
% &=&\sum_{i=0}^e \binom{e}{i}^2 (x-t)^{e-i}t^i(x-1)^i\\
% &=&\sum_{i=0}^e \binom{e}{i}^2 t^i
% \sum_{j=0}^{e-i} \binom{e-i}{j}x^j(-t)^{e-i-j}
% \sum_{k=0}^{i} \binom{i}{k}x^k(-1)^{i-k}\\
% &=& \sum_{i=0}^e\sum_{j=0}^{e-i}\sum_{k=0}^{i}
% (-1)^{e-j-k}\binom{e}{i}^2\binom{e-i}{j}\binom{i}{k}
% x^{j+k}t^{e-j},
% \end{eqnarray*}
\begin{eqnarray*}
P&=&\sum_{i=0}^e \binom{e}{i}^2 (x-t)^{e-i}t^i(x-1)^i = \sum_{i=0}^e \binom{e}{i}^2 t^i
\sum_{j=0}^{e-i} \binom{e-i}{j}x^j(-t)^{e-i-j}
\sum_{k=0}^{i} \binom{i}{k}x^k(-1)^{i-k}\\
&=& \sum_{i=0}^e\sum_{j=0}^{e-i}\sum_{k=0}^{i}
(-1)^{e-j-k}\binom{e}{i}^2\binom{e-i}{j}\binom{i}{k}
x^{j+k}t^{e-j},
\end{eqnarray*}
whose coefficient of $x^m t^n$ is
\begin{eqnarray}
\sum_{i=m+n-e}^n (-1)^{e-m}
\binom{e}{i}^2\binom{e-i}{e-n}\binom{i}{m+n-e}. \label{eq:coef_of_right}
\end{eqnarray}
From \eqref{eq:coef_of_left} and \eqref{eq:coef_of_right}, it suffices to show
\[
\frac{e!}{(p-1)!}\frac{(e+m)!}{m!}\binom{e}{m+n-e}\binom{e}{e-n}
=\sum_{i=m+n-e}^n 
\binom{e}{i}^2\binom{e-i}{e-n}\binom{i}{m+n-e}.
\]
Since
\[
\binom{e}{i}\binom{e-i}{e-n}=\frac{e!}{i!(e-n)!(n-i)!}=\binom{e}{e-n}\binom{n}{i}
\]
and similarly
\[
\binom{e}{i}\binom{i}{m+n-e}=\frac{e!}{(e-i)!(m+n-e)!(i-m-n+e)!}
=\binom{e}{m+n-e}\binom{2e-m-n}{e-i},
\]
it suffices to show that
\begin{eqnarray}
\frac{e!}{(p-1)!}\frac{(e+m)!}{m!}
& = &\sum_{i=m+n-e}^n \binom{n}{i}\binom{2e-m-n}{e-i} \label{eq:final}
\end{eqnarray}
which we show in the following by rearranging the right hand into the left hand side.
Putting $k=i-(m+n-e)$, one has
\begin{eqnarray*}
\sum_{i=m+n-e}^n 
\binom{n}{i}\binom{2e-m-n}{e-i}
&=& \sum_{k=0}^{e-m} \binom{n}{k+m+n-e}\binom{2e-m-n}{2e-m-n-k} .
\end{eqnarray*}
% \begin{eqnarray*}
% 
% &=& \sum_{k=0}^{e-m} 
% \binom{2e-m-n}{k}\binom{(2e-m)-(2e-m-n)}{(e-m)-k},
% \end{eqnarray*}
Since $\binom{n}{k+m+n-e} = \binom{n}{n - (k+m+n-e)} = \binom{n}{(e-m)-k}$, we also have
\begin{eqnarray*}
\sum_{i=m+n-e}^n 
\binom{n}{i}\binom{2e-m-n}{e-i}
&=& \sum_{k=0}^{e-m} 
\binom{2e-m-n}{k}\binom{n}{(e-m)-k},
\end{eqnarray*}
the right hand side of which is equal to
\begin{equation}\label{result-by-Vandermonde}
\binom{2e - m}{e -m} = \frac{(2e-m)!}{(e-m)!e!}
\end{equation}
by Vandermonde's identity.
By Lemma \ref{lem:factorial}, we have
$(2e-m)! = (-1)^{m+1}/m!$ and $(e-m)! = (-1)^{e+m+1}/(e+m)!$ and
$e!=(-1)^{e+1}/e!$, whence
\eqref{result-by-Vandermonde} is equal to
\[
\frac{e!}{(p-1)!}\frac{(e+m)!}{m!}
\]
modulo $p$ by Wilson's theorem.
This is equal to the left hand side of \eqref{eq:final}.
%
%It follows from Lemma \ref{lem:factorial} that the right hand side of $\binom{2e - m}{e -m}=\frac{(2e-m)!}{(e-m)! e! }$ is equal to $\frac{(2e-m)!(e+m)!}{(-1)^{-e-m-1} e! }$ modulo $p$.
%Taking $m+1$ as $i$ in Lemma \ref{lem:factorial}, we have $(2e-m)! = (p-m-1)! \equiv \frac{(-1)^{m+1}}{m!} \pmod p$, and thus
%\[
%\binom{2e - m}{e -m} \equiv  \frac{(2e-m)!(e+m)!}{(-1)^{-e-m-1} e! } \equiv \frac{(-1)^{m+1}}{(-1)^{-e-m-1} m!} \cdot \frac{(e+m)!}{e!} \equiv (-1)^{\frac{p-1}{2}} \frac{(e+m)!}{e! m!} \pmod p.
%\]
%Finally, it follows from Lemma \ref{lem:factorial} together with Wilson's theorem that
%\[
%(-1)^{\frac{p-1}{2}} \frac{(e+m)!}{e! m!} = \frac{(-1)^{\frac{p+1}{2}}}{-1} \frac{(e+m)!}{e! m!}  \equiv \frac{e!}{(p-1)!}\frac{(e+m)!}{m!} \pmod p,
%\]
%which is equal to the left hand side of \eqref{eq:final}.
%\if 0
%\begin{eqnarray*}
%\binom{2e - m}{e -m}
%&=& \frac{(2e-m)!}{(e-m)! e! } \\
%&=& \frac{(2e-m)!(e+m)!}{(-1)^{-e-m-1} e! } \\
%&=& \frac{(2e-m)!}{(-1)^{-e-m-1}} \frac{(e+m)!}{e!}\\
%&=& \frac{(p-m-1)!}{(-1)^{-e-m-1}} \frac{(e+m)!}{e!}\\
%&=& \frac{(-1)^{m+1}}{(-1)^{-e-m-1} m!} \frac{(e+m)!}{e!}\\
%&=& (-1)^{\frac{p-1}{2}} \frac{(e+m)!}{e! m!}.
%\end{eqnarray*}
%In other hands,
%\begin{eqnarray*}
%\frac{e!}{(p-1)!}\frac{(e+m)!}{m!} 
%&=& \frac{e! (e+m)!}{ (-1) m!} \\
%&=& \frac{(e!)^2 (e+m)!}{(-1) e! m!} \\
%&=& \frac{(-1)^{\frac{p+1}{2}} (e+m)!}{(-1) e! m!} \\
%&=& (-1)^{\frac{p-1}{2}} \frac{(e+m)!}{e! m!}.
%\end{eqnarray*}
%\fi
\end{proof}

\begin{lem}\label{KeyLem}
Let $K$ be an algebraically closed field of characteristic $p$.
Let $y^2 = g_0 (x)$ be an elliptic curve over $K$, where $g_0 (x)$ is a cubic polynomial in $K[x]$.
Let $\epsilon_i$ and $\epsilon_i'$ be the $x^i$-coefficients of $g_0(x)^{(p-1)/2}$ and $(r g_0(ux+v) )^{(p-1)/2}$ respectively for $0 \leq i \leq 3 (p-1)/2$, where $r, u\in K^\times$ and $v\in K$.
Then we have the following:
\begin{enumerate}
\item[$(1)$] $\epsilon_{p-1} = 0$ if and only if $\epsilon_{p-1}^{\prime} = 0$.
\item[$(2)$] If $\epsilon_{p-1} = 0$ $($or equivalently $\epsilon_{p-1}^{\prime}= 0$$)$, then $\epsilon_{p-2}\ne 0$ if and only if $\epsilon_{p-2}^{\prime} \ne 0 $.
\item[$(3)$] If $\epsilon_{p-1} = 0$ $($or equivalently $\epsilon_{p-1}^{\prime}= 0$$)$, then $\epsilon_{p-2} \neq 0$ $($and hence $\epsilon_{p-2}' \neq 0$$)$.
\end{enumerate}
\end{lem}

\begin{proof}
First we show $(1)$ and $(2)$.
Similarly to the proof of Lemma \ref{ChangeOfVariables}, we have
\[
\epsilon_{i}'  = r^{\frac{p-1}{2}} \sum_{j=i}^{3 (p-1)/2} \binom{j}{i} u^i v^{j-i} \epsilon_{j} 
\]
for $0 \leq i \leq 3 (p-1)/2$, and in particular $\epsilon_{p-1}^{\prime}$ and $\epsilon_{p-2}^{\prime}$ are
\begin{eqnarray}
\epsilon_{p-1}' &=& r^{\frac{p-1}{2}} u^{p-1} \epsilon_{p-1} , \nonumber \\
\epsilon_{p-2}' &=& r^{\frac{p-1}{2}} u^{p-2} (\epsilon_{p-2} - v \epsilon_{p-1} ) \nonumber 
\end{eqnarray}
from which (1) and (2) follows.

Next, we show (3).
Assume $\epsilon_{p-1} = 0$.
Since $y^2 = g_0 (x)$ is isormorphic to a Legendre elliptic curve, there exist elements $r^{\prime} \neq 0$, $u^{\prime} \neq 0$ and $v^{\prime}$ in $K$ such that $g_0 (x) = r^{\prime} g (X)$ with $X := u' x  + v'$, where $g(X)=X(X-1)(X-t)$ for some $t \in K \smallsetminus \{ 0, 1 \}$.
By (1) and our assumption $\epsilon_{p-1} = 0$, the $X^{p-1}$-coefficient of $g(X)^{(p-1)/2}$ is zero, and thus the $X^{p-2}$-coefficient of $g (X)^{(p-1)/2}$ is not zero by Proposition \ref{KeyProp} (4).
It follows from (2) that $\epsilon_{p-2} \neq 0$, and we also have $\epsilon_{p-2}^{\prime} \neq 0$.
\end{proof}

\section{Proof of the main theorem}
Assume $p>3$.
Let $K$ be an algebraically closed field of characteristic $p$.
We use the same notation as in Section 2, i.e.,
$A_1,B_1,A_2,B_2$, $\lambda, \mu, \nu$, $f_1, f_2, f$, $a,b,c,d$ are 
as in Section 2.
We choose $A_1,B_1,A_2,B_2$ in $K$ so that
\eqref{originalE1} and \eqref{originalE2} are supersingular.
By Proposition \ref{EquationsOfSupersingularity},
it suffices to show that there exists $(\lambda,\mu,\nu)\in K^3$ 
of Howe type such that $ad-bc=0$, $ab^{p-1}+d^p=0$ and $a^p+c^{p-1}d=0$.

From now on, we regard $\lambda$, $\mu$ and $ \nu$ as intererminates,
and consider $a,b,c,d$ as polynomials in $\lambda$, $\mu$, $\nu$.
Note that $a,b,c$ and $d$ are homogeneous polynomials in $\lambda,\mu,\nu$
of degrees $2p-2$, $p-2$, $2p-1$ and $p-1$ respectively

Write $f_1^{\frac{p-1}{2}} = \sum_{k=0}^{3(p-1)/2} \alpha_k x^k$ and $f_2^{\frac{p-1}{2}} = \sum_{k=0}^{3(p-1)/2} \beta_k x^k$.
Putting
\begin{eqnarray}
F_2 (x) &:=& x^3 + A_2 \nu^2 x + B_2 \nu^3 , \label{eq:F2}
\end{eqnarray}
we can write
\[
f_2^{\frac{p-1}{2}} = (F_2^{\frac{p-1}{2}}) (x - \lambda) = \sum_{k=0}^{\frac{3}{2}(p-1)} \frac{1}{k !} \left( F_2^{\frac{p-1}{2}} \right)^{(k)} ( - \lambda ) x^k 
\]
and thus
\begin{eqnarray}
\beta_{k} & = & \frac{1}{k !} \left( F_2^{\frac{p-1}{2}} \right)^{(k)} ( - \lambda ) . \label{eq:beta_k}
\end{eqnarray}
Moreover, we denote by $\beta_k^{\prime}$ the $x^k$-coefficient of $F_2^{\frac{p-1}{2}}$, i.e., $F_2^{\frac{p-1}{2}} = \sum_{k=0}^{3(p-1)/2} \beta_k^{\prime} x^k$ with $\beta_{3(p-1)/2}^{\prime}=1 $.
Note that each coefficient $\beta_k^{\prime}$ is a term in $K[\nu]$ of degree $3 e - k$ since $F_2^{\frac{p-1}{2}}$ is homogeneous of degree $3 e$ as a polynomial in $K[x, \nu] $.
We also have
\[
\left( F_2^{\frac{p-1}{2}} \right)^{(k)} = \sum_{n=k}^{\frac{3}{2}(p-1) - k} \frac{n !}{\left( n - k \right) !} \beta_{n}^{\prime} x^{n-k}.
\]

\begin{lem}\label{lem:beta01}
The coefficient $\beta_0$ $($resp.\ $\beta_1)$ is a homogeneous polynomial in $K[\lambda, \nu]$ of degree $3 e$ $($resp.\ $3e - 1)$ with $e = (p-1)/2$, and its highest term in $\lambda$ is $(-1)^{3 e} \lambda^{3 e}$ $($resp.\ $(3 e - 1) (-1)^{3 e -1} \lambda^{3 e - 1})$.
\end{lem}
\begin{proof}
It follows from \eqref{eq:F2} and \eqref{eq:beta_k} that
\[
\beta_0 = \left\{ ( - \lambda )^3 + A_2 \nu^2 ( - \lambda ) + B_2 \nu_3 \right\}^{\frac{p-1}{2}} ,
\]
which shows that $\beta_0$ is homogeneous of degree $3 e$, and that its highest term in $\lambda$ is $(-1)^{3 e} \lambda^{3 e}$.
We also have
\[
\beta_1 = \sum_{n=1}^{\frac{3}{2}(p-1)} n \beta_n^{\prime} ( - \lambda )^{n-1} ,
\]
which is homogeneous of degree $3 e - 1$ if it is not zero.
It follows from $\beta_{3(p-1)/2}^{\prime}=1 $ that the highest term of $\beta_1$ in $\lambda$ is 
% Note that $c_i (p-1)=0$ since $E_1$ and $E_2$ are supersingular.
$\frac{3}{2}(p-1) (-1)^{\frac{3}{2}(p-1)-1} \lambda^{\frac{3}{2}(p-1)-1}$, which is not zero since $p > 3$.
\end{proof}

\begin{lem}\label{lem:beta_e}
The coefficient $\beta_e$ $($resp.\ $\beta_{e+1})$ with $e=(p-1)/2$ is a homogeneous polynomial in $K[\lambda, \nu]$ of degree $2 e$ $($resp.\ $2e - 1)$.
The highest terms of $\beta_e$ and $\beta_{e+1}$ in $\lambda$ are
\[
 \frac{ (p-2) !}{\left( \cfrac{p-1}{2} \right) ! \left( \cfrac{p-3}{2} \right) !} \beta_{p-2}^{\prime} ( - 1)^{\frac{p-3}{2}} \lambda^{\frac{p-3}{2}}, \quad  \mbox{and} \quad
\frac{ (p-2) !}{\left( \cfrac{p+1}{2} \right) ! \left( \cfrac{p-5}{2} \right) !} \beta_{p-2}^{\prime} ( - 1)^{\frac{p-5}{2}} \lambda^{\frac{p-5}{2}},
\]
respectively.
\end{lem}
\begin{proof}
It follows from \eqref{eq:beta_k} that
\begin{eqnarray}
\beta_{e} & = & 
\frac{1}{\left( \frac{p-1}{2} \right) !}\sum_{n=\frac{p-1}{2}}^{\frac{3}{2}(p-1) - \frac{1}{2}(p-1)} \frac{n !}{\left( n- \frac{p-1}{2} \right) !} \beta_{n}^{\prime} ( - \lambda )^{n-\frac{p-1}{2}} \nonumber \\
& = & \sum_{n=\frac{p-1}{2}}^{p-1} \frac{n !}{\left( \frac{p-1}{2} \right) ! \left( n- \frac{p-1}{2} \right) !} \beta_{n}^{\prime} ( - 1)^{n-\frac{p-1}{2}} \lambda^{n-\frac{p-1}{2}} , \nonumber
\end{eqnarray}
which is homogeneous of degree $2 e$ in $K[\lambda, \nu]$ since each $\beta_n^{\prime}$ is homogeneous of degree $3e - n$ in $K[\nu]$.
Recall that $E_2$ is supersingular, and thus $\beta_{p-1}^{\prime} = 0$ and $\beta_{p-2}^{\prime} \neq 0$ by Lemma \ref{KeyLem}.
Hence the highest term of $\beta_e$ with respect to $\lambda$ is
\begin{eqnarray}
 \frac{ (p-2) !}{\left( \cfrac{p-1}{2} \right) ! \left( \cfrac{p-3}{2} \right) !} \beta_{p-2}^{\prime} ( - 1)^{\frac{p-3}{2}} \lambda^{\frac{p-3}{2}} , \label{eq:d_lambda0}
\end{eqnarray}
and it is not zero.

We also have
\begin{eqnarray}
\beta_{e+1} & = & 
\frac{1}{\left( \frac{p+1}{2} \right) !}\sum_{n=\frac{p+1}{2}}^{\frac{3}{2}(p-1) - \frac{1}{2}(p+1)} \frac{n !}{\left( n- \frac{p+1}{2} \right) !} \beta_{n}^{\prime} ( - \lambda )^{n-\frac{p+1}{2}} \nonumber \\
& = & \sum_{n=\frac{p+1}{2}}^{p-2} \frac{n !}{\left( \frac{p+1}{2} \right) ! \left( n- \frac{p+1}{2} \right) !} \beta_{n}^{\prime} ( - 1)^{n-\frac{p+1}{2}} \lambda^{n-\frac{p+1}{2}} , \nonumber
\end{eqnarray}
which is homogeneous of degree $2 e - 1$ if it is not zero.
Since $\beta_{p-2}^{\prime} \neq 0$, the highest term of $\beta_{e+1}$ with respect to $\lambda$ is
\begin{eqnarray}
 \frac{ (p-2) !}{\left( \cfrac{p+1}{2} \right) ! \left( \cfrac{p-5}{2} \right) !} \beta_{p-2}^{\prime} ( - 1)^{\frac{p-5}{2}} \lambda^{\frac{p-5}{2}} , \label{eq:b_lambda0}
\end{eqnarray}
which is not zero. 
% modulo $p > 3$.
\end{proof}

\begin{lem}\label{Orders}
We have the following:
\begin{enumerate}
\item[\rm (1)] $\ord_\mu(a) = \frac{p+1}{2}$.
\item[\rm (2)] $\ord_\mu(c) = \frac{p+1}{2}$.
\item[\rm (3)] $\ord_\mu(ad-bc)=\ord_\nu(ad-bc)=\frac{p+1}{2}$.
\item[\rm (4)] $\frac{a d - b c}{(\mu \nu )^{\frac{p+1}{2}} } \equiv B \lambda^{2 p - 4} \bmod{(\mu, \nu )}$ for some constant $B \in K^{\times}$.
\end{enumerate}
\end{lem}
\begin{proof}
\begin{enumerate}
\item[\rm (1)] We claim that there exists $\tilde{\alpha}_{k} \in K$ such that $\alpha_k = \mu^{3e - k} \tilde{\alpha}_{k}$ for each $0 \leq k \leq 3(p-1)/2$.
Indeed, we have
\begin{eqnarray}
\displaystyle 
\alpha_{k} & = & \sum_{3 n_0 + n_1=k} \frac{e!}{{n_0}! {n_1}! (e - n_0 - n_1)!} (A_1 \mu^2)^{n_1} (B_1 \mu^3)^{e - n_0 - n_1} \nonumber \\
& = & \left( \sum_{3 n_0 + n_1=k} \frac{e!}{{n_0}! {n_1}! (e - n_0 - n_1)!} A_1^{n_1} B_1^{e - n_0 - n_1} \right) \mu^{3 e - k} \nonumber
\end{eqnarray}
with $e = (p-1)/2$, and thus $\alpha_{k}$ is divided by $\mu^{3 e - k}$.
Putting $\tilde{\alpha}_{k} := \alpha_{k} / (\mu^{3 e - k})$, we also have $\tilde{\alpha}_k \in K$ for $0 \leq k \leq 3(p-1)/2$.
% the degree of $\alpha_k$ on $\mu$ is equal to $3e - k$ if $\alpha_k \neq 0$ for $0 \le k \le 3(p-1)/2$.
Since both the elliptic curves $E_1 : z^2 = f_1$ and $E_2 : w^2 = f_2$ are supersingular, we have $\alpha_{p-1} = 0$ and $\beta_{p-1} = 0$, whereas $\alpha_{p-2} \neq 0$ (and thus $\tilde{\alpha}_{p-2} \neq 0$) and $\beta_{p-2} \neq 0$ by Lemma \ref{KeyLem}.
It follows from
\[
\displaystyle a = \gamma_{p-1} = \sum_{k=1}^{p-2} \alpha_{k} \beta_{p-1-k} = \sum_{k=1}^{p-2} \mu^{3 e - k} \tilde{\alpha}_{k} \beta_{p-1-k} 
= \sum_{j=\frac{p+1}{2}}^{\frac{3 p -5}{2}} \mu^{j} \tilde{\alpha}_{3 e - j} \beta_{j - e} ,
\]
where $\beta_k$ is a polynomial in $K [\lambda, \nu]$ for each $0 \leq k \leq 3 e$.
% ($\tilde{\alpha}_{k} \in K$.)
% If $1 \le k \le p-2$, then $\frac{p+1}{2} \le 3e -k \le \frac{3p-5}{2}$.
Since $\beta_1 \neq 0$ by Lemma \ref{lem:beta01}, we have $\tilde{\alpha}_{3 e - j} \beta_{j-e} = \tilde{\alpha}_{p-2} \beta_{1} \neq 0 $ for $j = (p+1)/2$, and thus $\ord_{\mu} (a) = (p+1)/2$.
%  is divided by $\mu^{\frac{p+1}{2}}$.
\item[\rm (2)]
Similarly to the proof of (1), one has
\[
\displaystyle c = \gamma_{p-2} = \sum_{k=0}^{p-2} \alpha_{k} \beta_{p-2-k} = \sum_{k=0}^{p-2} \mu^{3e - k} \tilde{\alpha}_{k} \beta_{p-2-k}
= \sum_{j=\frac{p+1}{2}}^{\frac{3 p - 3}{2}} \mu^{j} \tilde{\alpha}_{3 e - j} \beta_{j - e - 1}.
\]
Since we have $\tilde{\alpha}_{p-2}\neq 0$, and $\beta_0 \neq 0$ by Lemma \ref{lem:beta01}, we also have $\tilde{\alpha}_{3 e - j} \beta_{j-e-1} = \tilde{\alpha}_{p-2} \beta_{0} \neq 0 $ for $j = (p+1)/2$, and thus $\ord_{\mu} (c) = (p+1)/2$.
% If $0 \le k \le p-2$, then $\frac{p+1}{2} \le 3e -k \le \frac{3p-3}{2}$.
% Thus $a$ is divided by $\mu^{\frac{p+1}{2}}$.
\item[\rm (3)]
Similarly to the proof of (1), remaining two entries of the Cartier-Manin matrix $M$ are written as
\begin{eqnarray*}
b &=& \gamma_{2p-1} = \sum_{k=\frac{p+1}{2}}^{\frac{3}{2}(p-1)} \alpha_{k} \beta_{2p-1-k} = \sum_{k=\frac{p+1}{2}}^{\frac{3}{2}(p-1)} \mu^{3e - k} \tilde{\alpha}_{k} \beta_{2p-1-k}
= \sum_{j=0}^{p-2} \mu^{j} \tilde{\alpha}_{3 e - j} \beta_{j + e + 1}, \\
d &=& \gamma_{2p-2} = \sum_{k=\frac{p-1}{2}}^{\frac{3}{2}(p-1)} \alpha_{k} \beta_{2p-2-k} = \sum_{k=\frac{p-1}{2}}^{\frac{3}{2}(p-1)} \mu^{3e - k} \tilde{\alpha}_{k} \beta_{2p-2-k}
=\sum_{j=0}^{p-1} \mu^{j} \tilde{\alpha}_{3e - j} \beta_{j+e},
\end{eqnarray*}
both of which are not divided by $\mu$ since $\alpha_{3 e} = 1$ and since $\beta_{e}, \beta_{e+1} \neq 0$ by Lemma \ref{lem:beta_e}.
Thus, if the coefficient of $\mu^{(p+1)/2}$ in $ad-bc$ is not zero, then we have $\ord_{\mu}(a d - b c) = (p+1)/2$.
% \textcolor{red}{Since $\mu^{3e - k}$ is equal to 1 if $k = \frac{p-1}{2}$, $b$ and $d$ are not divided by $\mu$.
% Thus $ad - bc$ is divided by $\mu^{\frac{p+1}{2}}$.}
By straightforward computation, the coefficients of $\mu^{(p+1)/2}$ in $ad$ and $bc$ are $\tilde{\alpha}_{p-2} \beta_1 \tilde{\alpha}_{3 e} \beta_e$ and $\tilde{\alpha}_{p-2} \beta_0 \tilde{\alpha}_{3 e} \beta_{e+1}$, respectively.
Here we have
\[
\tilde{\alpha}_{p-2} \beta_1 \tilde{\alpha}_{3 e} \beta_e - \tilde{\alpha}_{p-2} \beta_0 \tilde{\alpha}_{3 e} \beta_{e+1} 
= \tilde{\alpha}_{p-2} \tilde{\alpha}_{3 e} ( \beta_1 \beta_e - \beta_0 \beta_{e+1}),
\]
where $\tilde{\alpha}_{p-2} \neq 0$ and $\tilde{\alpha}_{3 e} = \alpha_{3 e} = 1$.
If $\beta_1 \beta_e - \beta_0 \beta_{e+1} \neq 0$, we have $\ord_{\mu}(ad - b c)= (p+1)/2$.
% Recall from the beginning of this section that $\beta_{k}$ is homogeneous of degree $3 e - k$ in $K[\lambda, \nu]$, 
By Lemmas \ref{lem:beta01} and \ref{lem:beta_e}, the highest term of $\beta_1 \beta_{e}$ in $\lambda$ is
\begin{eqnarray}
& & \frac{3}{2}(p-1) (-1)^{\frac{3 p-5}{2}} \lambda^{\frac{3 p-5}{2}} \frac{ (p-2) !}{\left( \cfrac{p-1}{2} \right) ! \left( \cfrac{p-3}{2} \right) !} \beta_{p-2}^{\prime} ( - 1)^{\frac{p-3}{2}} \lambda^{\frac{p-3}{2}} \nonumber \\
 & = &  \frac{3}{2}(p-1) \frac{ (p-2) !}{\left( \cfrac{p-1}{2} \right) ! \left( \cfrac{p-3}{2} \right) !} \beta_{p-2}^{\prime} (-1)^{2 p - 4} \lambda^{2 p - 4}, \nonumber
\end{eqnarray}
and that of $\beta_{0} \beta_{e+1}$ is
\[
(-1)^{\frac{3 p - 3}{2}} \lambda^{\frac{3 p - 3}{2}} \frac{ (p-2) !}{\left( \cfrac{p+1}{2} \right) ! \left( \cfrac{p-5}{2} \right) !} \beta_{p-2}^{\prime} ( - 1)^{\frac{p- 5}{2}} \lambda^{\frac{p- 5}{2}}
 =  \frac{ (p-2) !}{\left( \cfrac{p+1}{2} \right) ! \left( \cfrac{p -5}{2} \right) !} \beta_{p-2}^{\prime} (-1)^{2 p - 4} \lambda^{2 p - 4} . 
\]
Since $\beta_{p-2}' \neq 0$ by Lemma \ref{KeyLem}, it suffices to show 
\[
\frac{3}{2}(p-1) \frac{ (p-2) !}{\left( \cfrac{p-1}{2} \right) ! \left( \cfrac{p - 3}{2} \right) !}
\neq \frac{ (p-2) !}{\left( \cfrac{p+1}{2} \right) ! \left( \cfrac{p - 5}{2} \right) !} \pmod{p}.
\]
If the equality holds, we have
\[
\frac{\frac{3}{2}(p-1)}{\frac{1}{2}(p-3)} = \frac{1}{\frac{1}{2}(p+1)} \pmod {p}
\]
%and thus $1 \equiv 2^{-1} \pmod{p}$.
and thus $1 \equiv 2 \pmod{p}$.
This is a contradiction.
% Hence the highest term of $a d - b c$ with respect to $\lambda$ is $B \lambda^{2 p - 4}$ for some constant $B \in K^{\times}$, and so
% \[
% \frac{a d - b c}{(\mu \nu )^{\frac{p+1}{2}} } \equiv B \lambda^{2 p - 4} \bmod{(\mu, \nu )} .
% \]
% and thus $\beta_{1} \beta_e$ and $\beta_0 \beta_{e-2}$ are homogeneous of degrees $5 e  - 1$ and $5 e + 2$.

\indent Next, we show $\ord_{\nu} (a d - b c) = (p+1)/2$.
To show this, we consider the transformation $X = x - \lambda$.
% that the polynomial $ad -bc$ does not change by the transformation $X = x - \lambda$.
The polynomials $f_1 (x)$ and $f_2 (x)$ are rewritten as
\begin{eqnarray*}
f_1^{\prime} (X)  &:= & (X + \lambda)^3 + A_1 \mu^2 (X + \lambda )+ B_1\mu^3, \\
f_2^{\prime} (X)  &:= & X^3 + A_2 \nu^2X + B_2 \nu^3
\end{eqnarray*}
respectively.
% we have
%\[ 
%\displaystyle (f_1(x) f_2(x))^{\frac{p-1}{2}} = 
%(f_1^{\prime}(X) f_2^{\prime}(X))^{\frac{p-1}{2}} = 
%\sum_{l=0}^{3(p-1)} \gamma^{\prime}_l X^l = 
%\sum_{l=0}^{3(p-1)} \sum_{k = 0}^{l} 
%\binom{l}{k} \gamma^{\prime}_l (-\lambda)^{l-k} x^k, \]
%where $\gamma_{l}^{\prime}$ denotes 
%the coefficient of $X^{l}$ in $(f_1^{\prime}(X) f_2^{\prime}(X))^{(p-1)/2}$.
Let $\gamma_{l}^{\prime}$ denote
the coefficient of $X^{l}$ in $(f_1^{\prime}(X) f_2^{\prime}(X))^{(p-1)/2}$.
% Since $(f_1(x) f_2(x))^{(p-1)/2} = (f_1(X) f_2(X))^{(p-1)/2}$, we have
%Recall from (10) that we have
%\[ \gamma_k = \sum_{l=k}^{3(p-1)} \binom{l}{k} \gamma^{\prime}_l (-\lambda)^{l-k}, \]
%and hence
%\begin{eqnarray*}
%a &=& \gamma_{p-1} = \sum_{l=p-1}^{3(p-1)} \binom{l}{p-1} \gamma^{\prime}_l (-\lambda)^{l-(p-1)} = \gamma^{\prime}_{p-1}, \\
%c &=& \gamma_{p-2} = \sum_{l=p-2}^{3(p-1)} \binom{l}{p-2} \gamma^{\prime}_l (-\lambda)^{l-(p-2)} = \gamma^{\prime}_{p-2} + \lambda \gamma^{\prime}_{p-1}, \\
%b &=& \gamma_{2p-1} = \sum_{l=2p-1}^{3(p-1)} \binom{l}{2p-1} \gamma^{\prime}_l (-\lambda)^{l-(2p-1)} = \gamma^{\prime}_{2p-1}, \\
%d &=& \gamma_{2p-2}  = \sum_{l=2p-2}^{3(p-1)} \binom{l}{2p-2} \gamma^{\prime}_l (-\lambda)^{l-(2p-2)} = \gamma^{\prime}_{2p-2} + \lambda \gamma^{\prime}_{2p-1}.
%\end{eqnarray*}
Putting $a^{\prime} = \gamma^{\prime}_{p-1}$, $c^{\prime} = \gamma^{\prime}_{p-2}$, $b^{\prime} = \gamma^{\prime}_{2p-1}$ and $d^{\prime} = \gamma^{\prime}_{2p-2}$, we have $ad - bc =a^{\prime} d^{\prime} -b^{\prime} c^{\prime}$
by Lemma \ref{ChangeOfVariables}.
%\[ ad - bc =
%\gamma^{\prime}_{p-1}
%( \gamma^{\prime}_{2p-2} + \lambda \gamma^{\prime}_{2p-1})
%-  \gamma^{\prime}_{2p-1} (\gamma^{\prime}_{p-2} 
%+ \lambda \gamma^{\prime}_{p-1}) = 
%a^{\prime} d^{\prime} -b^{\prime} c^{\prime} \]
By the same argument as the proof of $\ord_{\mu}(ad - bc) = (p+1)/2$ for $f_1(x)$ and $f_2(x)$, we have $\ord_{\nu}(a' d' - b' c') = (p+1)/2$ for $f_1^{\prime} (X)$ and $f_2^{\prime}(X)$, and thus $\ord_{\nu}(ad-bc) = (p+1)/2$.

\item[\rm (4)] From the first part of the proof of (3), the coefficient of $(\mu \nu )^{\frac{p+1}{2}}$ in $a d - b c$ is $B \lambda^{2 p - 4} $ with
\[
B:= \tilde{\alpha}_{p-2} \tilde{\alpha}_{3 e} ( \beta_{p-2}^{\prime} / \nu^{\frac{p+1}{2}} )
\left(  \frac{ \frac{3}{2}(p-1) \cdot (p-2) !}{\left( \cfrac{p-1}{2} \right) ! \left( \cfrac{p-3}{2} \right) !}  
- \frac{ (p-2) !}{\left( \cfrac{p+1}{2} \right) ! \left( \cfrac{p -5}{2} \right) !} \right) (-1)^{2 p - 4} ,
\]
which is not zero.
Recall from the proof of (1) that $\tilde{\alpha}_{p-2}$ and $\tilde{\alpha}_{3 e}$ are non-zero constants in $k$.
Recall also from the beginning of this section that $\beta_{p-2}^{\prime}$ is a term in $K[\nu]$ of degree $3 e - (p-2) = (p+1)/2$, and thus $( \beta_{p-2}^{\prime} / \nu^{\frac{p+1}{2}} ) \in K^\times$.
Thus the claim holds.
\end{enumerate}
\end{proof}

Let $R:=K[\lambda,\mu,\nu]$ and put
\begin{eqnarray*}
h_0 &:=& (ad-bc)/(\mu\nu)^{(p+1)/2},\\
h_1 &:=& ab^{p-1}+d^p,\\
h_2 &:=& a^p+c^{p-1}d,
\end{eqnarray*}
which belong to $R$.
Since $a,b,c$ and $d$ are homogeneous polynomials in $\lambda,\mu,\nu$
of degrees $2p-2$, $p-2$, $2p-1$ and $p-1$ respectively,
%Setting $\deg(\lambda)=\deg(\mu)=\deg(\nu)=1$, we have
%$\deg(\gamma_i)=3(p-1)-i$ and in particular
%\[
%\deg(a)=2p-2,\quad
%\deg(b)=p-2,\quad
%\deg(c)=2p-1,\quad
%\deg(d)=p-1.
%\]
we have that $h_0$, $h_1$ and $h_2$ are homogeneous
of degrees
$2p-4$, $p(p-1)$ and $2p(p-1)$ respectively.
Consider the homogeneous ideal
\[
I = \langle h_0, h_1, h_2 \rangle
\]
of $K[\lambda,\mu,\nu]$.

\begin{lem}\label{mu_or_nu_are_not_zero}
Assume $p>3$. For any point $(\lambda:\mu:\nu)$ on $V(h_0)$ in $\bbP^2$,
we have $\mu\ne 0$ or $\nu\ne 0$.
\end{lem}
\begin{proof}
It suffices to show that
$\mu=\nu=0$ implies $\lambda=0$
for $(\lambda, \mu, \nu)\in K^3$ satisfying $h_0=0$.
This immediately follows from Lemma \ref{Orders} (4).
\end{proof}

Consider the point with $\nu\ne 0$.
From now on,
we substitute $1$ for $\nu$
and consider $a,b,c,d$ as polynomials in $\lambda, \mu$.
%, and also consider $a_0,c_0, b_0,d_0$ as polynomials in $\lambda$.
Set
\[
S:=K[\lambda,\mu]
\]
and let $J$ be the ideal of $S$ obtained from $I$
by substituting $1$ for $\nu$.
Set $a'=a/\mu^{(p+1)/2}$ and $c'=c/\mu^{(p+1)/2}$.
Let $a'_0,c'_0, b_0,d_0$ be the constant terms of
$a',c', b,d$ as polynomials in $\mu$,
which are polynomials in $\lambda$.
\begin{lem}\label{KeyLemma}
As polynomials in $\lambda$, we have
\begin{enumerate}
\item[\rm (1)] $c'_0$ and $d_0$ are coprime.
\item[\rm (2)] $b_0$ and $d_0$ are coprime.
\end{enumerate}
\end{lem}
\begin{proof}
(1) Since
\[
c'_0=\tilde\alpha_{p-2}\beta_0,\qquad
d_0=\tilde\alpha_{3e}\beta_e
\]
and $\tilde \alpha_{p-2}$ and $\tilde \alpha_{3e}$ are non-zero constant as polynomials in $\lambda$, it suffices to see that $\beta_0$ and $\beta_e$
are coprime as polynomials in $\lambda$.
As
\[
\beta_0 = F_2^{e}(-\lambda),\qquad
\beta_e = \frac{1}{e!}(F_2^{e})^{(e)}(-\lambda)
\]
and $F_2$ is a separable polynomial, $\beta_0$ and $\beta_e$
are coprime.

(2) Since
\[
b_0 = \tilde\alpha_{3e}\beta_{e+1},\qquad
d_0=\tilde\alpha_{3e}\beta_e
\]
and $\tilde \alpha_{3e}$ is a non-zero constant,
it suffices to see that $\beta_e$ and $\beta_{e+1}$ are coprime.
As
\[
\beta_e = \frac{1}{e!}(F_2^{e})^{(e)}(-\lambda),\qquad
\beta_{e+1} = \frac{1}{(e+1)!}(F_2^{e})^{(e+1)}(-\lambda),
\]
it suffices to show that $(F_2^{e})^{(e)}(-\lambda)$
is a separable polynomial.
A linear coordinate change makes $F_2$ a Legendre form $g(x)=x(x-1)(x-a_1)$,
where $a_1$ is a solution of $H_p(t)=0$ as $y^2=F_2$ is a supersingular elliptic curve. Let $H_p(t) = \prod_{i=1}^e (t-a_i)$ be a factorization of $H_p(t)$.
By Proposition \ref{Separatedness}, up to multiplication by a non-zero constant,
$(g(x)^e)^{(e)}$  is factored as
\begin{equation}\label{Separatedness_supersingular-case}
%-(1-a_1)a_1 
\prod_{i=2}^e\{(a_1-a_i)x-(1-a_i)a_1\}.
\end{equation}
Here we note that $a_i\ne 0,1$ for $i=1,2,\ldots,e$
as $H_p(0)=1$ and $H_p(1)=(-1)^e$,
see the proof of \cite[Chap.~V. Theorem 4.1 (c)]{Sil}.
%$y^2=x(x-1)(x-a_i)$ is a nonsingular (supersingular) elliptic curve.
Then it is clear that \eqref{Separatedness_supersingular-case}
is a separated polynomial, and therefore $\beta_e$ is separated.
\end{proof}

\begin{prop}
Suppose that $p>3$. Then
any point $(\lambda:\mu:\nu)$ on $V(h_0,h_1,h_2)$ is of Howe type.
\end{prop}
\begin{proof}
Let $(\lambda:\mu:\nu)$ be a point of $V(h_0,h_1,h_2)$.
By Lemma \ref{mu_or_nu_are_not_zero}, we have $\mu\ne 0$ or $\nu \ne 0$.
Suppose $\nu\ne 0$. Substututing $1$ for $\nu$ and consider everything as
a polynomial in $\mu$ and $\lambda$.
It follows from $h_0=0$ and $h_1=0$ that
\begin{eqnarray*}
a'_0 d_0 - b_0 c'_0 &\equiv& 0 \quad (\mod\ \mu),\\
d_0^p &\equiv& 0 \quad (\mod\ \mu).
\end{eqnarray*}
This and Lemma \ref{KeyLemma} imply that $1\equiv 0 \quad (\mod\ \mu)$,
whence $\mu$ is unit in $S/J$.
Thus we have shown that $\nu\ne 0$ implies $\mu\ne 0$
for any point $(\lambda:\mu:\nu)$ on $V(h_0,h_1,h_2)$.
A similar argument shows that $\mu\ne 0$ implies that $\nu\ne 0$.
Hence we conclude that both of $\mu$ and $\nu$ are not zero.

It remains to show that $f_1$ and $f_2$ defined in \eqref{f1} and \eqref{f2} are coprime.
As $\mu$ and $\nu$ are not zero, $f_1$ and $f_2$ are separated polynomials.
Suppose that $f_1$ and $f_2$ were not coprime.
After taking a linear coordinate change, they are written as
$f_1 = x(x-1)(x-t_1)$ and $f_2 = x(x-t_2)(x-t_3)$. Then
\[
(f_1f_2)^e = x^{p-1}\{(x-1)(x-t_1)(x-t_2)(x-t_3)\}^e.
\]
Then $M=\begin{pmatrix}a & b\\ c & d\end{pmatrix}$ becomes
a upper triangular matrix (i.e., $c=0$) with $a = (t_1t_2t_3)^e$.
Since $f_1$ and $f_2$ are separated, we have $t_i\ne 0$ for $i=1,2,3$
and therefore $h_2 \ne 0$. This is a contradiction.
\end{proof}

Finally we show that $V(h_0,h_1,h_2)$ is not empty.
The fact is reminiscent of the quasi-affineness (cf.\ \cite[(6.5).\ Theorem]{Oort}) of Ekedahl-Oort strata in the case of the moduli space of principally polarized abelian varieties.
\begin{prop}
$V(h_0,h_1,h_2)\ne \emptyset$.
\end{prop}
\begin{proof}
We suppose that $V(h_0,h_1,h_2) = \emptyset$ were true.
Set $X:=V(h_0)=V(h_0)\smallsetminus V(h_0,h_1,h_2)$.
% the non-supersingular part of $V(h_0)$.
Note that $(a,d)=(0,0)$ is not allowed on $X$,
as $(a,d)=(0,0)$ implies $h_1=h_2=0$.
Consider the morphism
\[
\varphi: X \longrightarrow \bbP^1\smallsetminus \{(0:1)\}\simeq \bbA^1
\]
sending $(\lambda:\mu:\nu)$ to 
$(a^p+c^{p-1}d:c^{p-1}d)$
for the part of $a\ne 0$ and
$(ab^{p-1}+d^p:d^p)$
for the part of $d\ne 0$, which is well-defined by \eqref{BasicEqualities1} and \eqref{BasicEqualities2}.
Moreover the image of $\varphi$ consists of infinitely many points
from the following two facts.
Firstly $X$ is connected since it is a hypersurface in $\bbP^2$.
Secondly $(1:1)$ (resp. $(1:0)$) is an image of $\varphi$
since 
there exists $(\lambda:\mu:\nu)\in \bbP^2$ such that
$h_0=0$ and $a = 0$ (resp. $h_0=0$ and $d = 0$),
by \cite[Chap.1, Theorem 7.2]{Har}.
The existence of such a morphism as $\varphi$ implies that $X$ is not projective. This is a contradiction.
\end{proof}


\begin{thebibliography}{99}
\bibitem{CO}
Chai, C.-L. and Oort, F.:
\textit{Monodromy and irreducibility of leaves},
Ann. of Math. (2) {\bf 173} (2011), no. 3, 1359--1396.

\bibitem{Deuring}
Deuring, M.:
{\it Die Typen der Multiplikatorenringe elliptischer Funktionenk\"orper},
Abh. Math. Sem. Univ. Hamburg {\bf 14} (1941), no. 1, 197--272. 

\bibitem{Gonzalez}
Gonz\'alez, J.:
{\it Hasse-Witt matrices for the Fermat curves of prime degree},
Tohoku Math. J. (2) {\bf 49} (1997), no. 2, 149--163. 
%\bibitem{MagmaHP}
%Cannon, J., et al.: \textit{Magma A Computer Algebra System},
%School of Mathematics and Statistics, University of Sydney, 2016.\ \texttt{http://magma.maths.usyd.edu.au/magma/}
%
%\bibitem{Hess}
%Hess, F.: \textit{An algorithm for computing Isomorphisms of Algebraic Function Field}, Algorithmic number theory, 263--271, Lecture Notes in Comput. Sci., {\bf 3076}, Springer, Berlin (2004). 

\bibitem{H07}
Harashita, S.: \textit{Ekedahl-Oort strata and the first Newton slope strata},
J. Algebraic Geom. {\bf 16} (2007), no. 1, 171--199. 

\bibitem{Har}
Hartshorne, R.: {\it Algebraic Geometry}, GTM {\bf 52}, Springer-Verlag (1977)

%\bibitem{Hashimoto}
%Hashimoto K.:
%{\it Class numbers of positive definite ternary quaternion Hermitian forms},
%Proc. Japan Acad. Ser. A Math. Sci. {\bf 59} (1983), no. 10, 490--493.

%\bibitem{HI}
%Hashimoto, K. and Ibukiyama, T.:
%{\it On class numbers of positive definite binary quaternion Hermitian forms. II}, J. Fac. Sci. Univ. Tokyo Sect. IA Math. {\bf 28} (1981), no. 3, 695--699 (1982).

\bibitem{Howe}
Howe, E.: \textit{Quickly constructing curves of genus {\rm 4} with many points},, pp. 149--173 in: Frobenius Distributions: Sato-Tate and Lang-Trotter conjectures (D. Kohel and I. Shparlinski, eds.), Contemporary Mathematics 663, American Mathematical Society, Providence, RI (2016)

\bibitem{Ibukiyama}
Ibukiyama, T.:
{\it On rational points of curves of genus $3$ over finite fields},
Tohoku Math. J. {\bf 45} (1993), 311-329.

\bibitem{IKO}
Ibukiyama, T., Katsura, T. and Oort, F.:
{\it Supersingular curves of genus two and class numbers}, 
Compositio Math. {\bf 57} (1986), no. 2, 127--152. 

\bibitem{Igusa}
Igusa, J.-I.: \textit{Class number of a definite quaternion with prime discriminant}, Proc. Nat. Acad. Sci. U.S.A., 44:312--314, 1958.

\bibitem{K18}
Kudo, M.:
\textit{On the existence of superspecial nonhyperelliptic curves of genus 4},
arXiv:1804.09063 (2018).

%\bibitem{KH16}
%Kudo, M.\ and Harashita, S.: \textit{Superspecial curves of genus $4$ in small characteristic}, Finite Fields and Their Applications {\bf 45}, 131--169 (2017).

%\bibitem{KH17}
%Kudo, M.\ and Harashita, S.: \textit{Enumerating superspecial curve of genus {\rm 4} over prime fields}, arXiv:1702.05313 (2017).

%\bibitem{KH18}
%Kudo, M. and Harashita, S.:
%\textit{Superspecial Hyperelliptic Curves of Genus $4$ over Small Finite Fields}, In: Budaghyan L., Rodriguez-Henriquez F. (eds), Arithmetic of Finite Fields, WAIFI 2018, Lecture Notes in Computer Science, Vol. 11321, pp. 58-73, Springer, Cham, 2018.

\bibitem{LMPT}
Li, W., Montovan, E., Pries, R. and Tang, Y.:
\textit{Newton polygons of cyclic covers of the projective line branched at three points},
To appear in Research Directions in Number Theory, Women in Numbers IV.

\bibitem{Pries}
Pries, R.:
\textit{Current results on Newton polygons of curves},
To appear as Chapter 6 of Open problems in Arithmetic Algebraic Geometry.

\bibitem{Senda}
Senda, H.:
\textit{Enumerating superspecial Howe curves},
Preprint.

\bibitem{Serre1983}
Serre, J.-P.:
{\it Nombre des points des courbes algebrique sur $\F_q$},
S\'em. Th\'eor. Nombres Bordeaux (2)
1982/83, 22 (1983).

%\bibitem{Serre}
%Serre, J.-P.: {\it Rational points on curves over finite fields},
%Lectures given at Harvard University 1985.
%Notes by Fernando Q. Gouv\'ea.


\bibitem{Sil}
Silverman, J. H.:
\textit{The Arithmetic of Elliptic Curves},
Second edition. Graduate Texts in Mathematics, 106. Springer, Dordrecht, 2009.

\bibitem{GV}
van der Geer, G. and van der Vlugt, M.:
\textit{On the existence of supersingular curves of given genus},
J. Reine Angew. Math. {\bf 458} (1995), 53--61. 

\bibitem{Oort}
Oort, F.:
\textit{A stratification of a moduli space of abelian varieties},
Progress in Mathematics {\bf 195}, pp. 345--416, Birkh\"auser, Basel, 2001.

%\bibitem{OU}
%Oort, F. and Ueno, K.:
%{\it Principally polarized abelian varieties of dimension two or three are Jaco%bian varieties},
%J. Fac. Sci. Univ. Tokyo Sect. IA Math. {\bf 20} (1973), 377--381. 

\bibitem{Yui}
Yui, N.:
\textit{On the Jacobian varieties of hyperelliptic curves over fields of characteristic $p>2$},
J. Algebra {\bf 52} (1978), no. 2, 378--410. 
\end{thebibliography}
\end{document}